\documentclass[11pt]{amsart}
\usepackage{amsfonts,amsmath,amsthm,amssymb,epsfig}

\usepackage[left=0.8in,top=0.8in,right=0.8in,bottom=0.8in]{geometry}

\newtheorem{theorem}{Theorem}

\begin{document}
\title{A note on bounds for the cop number using tree decompositions}
\author[A. Bonato]{Anthony Bonato}
\address{Department of Mathematics, Ryerson University, Toronto, ON, Canada,
M5B 2K3}
\email{\texttt{abonato@ryerson.ca}}
\author[N.E. Clarke]{Nancy E. Clarke}
\address{Department of Mathematics \& Statistics, Acadia University, Wolfville, NS, Canada, B4P 2R6}
\email{\texttt{nancy.clarke@acadiau.ca}}
\author[S. Finbow]{Stephen Finbow}
\address{Department of Mathematics, Statistics, \& Computer Science, St. Francis Xavier University, Antigonish, NS, Canada,
B2G 2W5}
\email{\texttt{sfinbow@stfx.ca}}
\author[S. Fitzpatrick]{Shannon Fitzpatrick}
\address{Department of Mathematics \& Statistics, University of Prince Edward Island,  Charlottetown, PE, Canada, C1A 4P3}
\email{\texttt{sfitzpatrick@upei.ca}}
\author[M.E. Messinger]{Margaret-Ellen Messinger}
\address{Department of Mathematics \& Computer Science, Mount Allison University, Sackville, NB,
Canada, E4L 1E6}
\email{mmessinger@mta.ca}
\thanks{Supported by grants from NSERC and Ryerson}

\begin{abstract}
In this short note, we supply a new upper bound on the cop number in terms of tree decompositions. Our
results in some cases extend a previously derived bound on the cop number using treewidth.
\end{abstract}

\maketitle

\section{Introduction}

The game of Cops and Robbers (defined at the end of this section) is usually
studied in the context of the minimum number of cops needed to have a
winning strategy, or \emph{cop number}. The cop number (written $c(G)$ for a graph $G$) is a challenging
graph parameter for a variety of reasons, and establishing upper bounds for
this parameter are the focus of Meyniel's conjecture: the cop number of a
connected $n$-vertex graph is $O(\sqrt{n}).$ For additional background on
Cops and Robbers and Meyniel's conjecture, see the recent book~\cite{bonato}.

The following elegant upper bound was given in \cite{joret}:
\begin{equation}
c(G)\leq tw(G)/2+1,  \label{first}
\end{equation}%
where $tw(G)$ is the treewidth of $G.$ The bound (\ref{first}) is
tight if the graph has small treewidth (up to treewidth $5).$ Further, it
gives a simple proof that outerplanar graphs have cop number at most $2$ (first proved in \cite{clarke8}).

For many families of graphs, however, the bound (\ref{first}) is far from
tight; for example, for a positive integer $n$, a clique $K_{n}$ has treewidth $n-1,$ but is cop-win.
Similarly, Cartesian $n\times n$ grids $P_{n}\square P_{n}$ have cop number $%
2,$ but have treewidth $n.$

In this short note, we give a new bound on the cop number that exploits tree decompositions, and in some cases improves on (\ref{first}). The idea of the proof
of (\ref{first}) is to guard bags and use isometric paths to move cops from one bag to another.
We modify this approach, and our main tool is the notion of a retract, and a retract cover of a graph. See
Theorems~\ref{main1}, \ref{i}, and \ref{main2}. Besides giving the correct bounds for various families (such as
grids, cliques, and $k$-trees), our results give a new approach to bounding the cop number by exploiting properties of tree decompositions.

\subsection{Definitions and notation}

We consider only finite, reflexive, undirected graphs in the paper. For
background on graph theory, the reader is directed to \cite{diestel,west}.

The game of \emph{Cops and Robbers} was independently introduced in~\cite{nw,q} and the cop number was introduced in~\cite{af}. The game is played on a reflexive
graph; that is,
vertices each have at least one loop. Multiple edges are allowed, but make
no difference to the game play, so we always assume there is exactly one
edge between adjacent vertices. There are two players consisting of a set of
\emph{cops} and a single \emph{robber}. The game is played over a sequence
of discrete time-steps or \emph{rounds}, with the cops going first in round $%
0$ and then playing alternate time-steps. The cops and robber occupy
vertices; for simplicity, we often identify the player with the vertex they
occupy. We refer to the set of cops as $C$ and the robber as $R.$ When a
player is ready to move in a round they must move to a neighbouring vertex.
Because of the loops, players can \emph{pass}, or remain on their own
vertex. Observe that any subset of $C$ may move in a given round. The cops
win if after some finite number of rounds, one of them can occupy the same
vertex as the robber (in a reflexive graph, this is equivalent to the cop
landing on the robber). This is called a \emph{capture}. The robber wins if
he can evade capture indefinitely.

If we place a cop at each vertex, then the cops are guaranteed to win.
Therefore, the minimum number of cops required to win in a graph $G$ is a
well-defined positive integer, named the \emph{cop number} (or \emph{%
copnumber}) of the graph $G.$ We write $c(G)$ for the cop number of a graph $%
G$. In the special case $c(G)=1,$ we say $G$ is \emph{cop-win}.

An induced subgraph $H$ of $G$ is a \emph{retract} if there is a
homomorphism $f$ from $G$ onto $H$ so that $f(x)=x$ for $x\in V(H);$ that
is, $f$ is the identity on $H.$ The map $f$ is called a \emph{retraction}.
For example, each isometric path is a retract (as shown first in \cite{af}), as is each clique. Each
retract $H$ with retraction $f$ can be \emph{guarded} by a set of cops in
the following sense: if the robber is on $x$, then the cops play in $H$ as
if the robber were on $f(x)$. If there are a sufficient number of cops to
capture the image of the robber, then, after finitely many rounds, if the
robber entered $H$ he would be immediately caught. We denote the minimum
number of cops needed to guard $H$ by $\mathrm{guard}(H)$. Note this
parameter is well-defined, as each vertex of $H$ can be guarded. In the case
of an isometric path $P$, it was shown in \cite{af} that $\mathrm{guard}(P)=1.$

\section{Tree decompositions}

In a tree decomposition, each vertex of the graph is represented by a
subtree, such  that vertices are adjacent only when the corresponding
subtrees intersect. Formally, given a graph $G = (V, E)$, a \textit{tree
decomposition} is a pair $(X, T)$, where $X = \{X_1, \ldots, X_n\}$ is a
family of subsets of $V$ called \emph{bags}, and $T$ is a tree whose vertices are the subsets $%
X_i$, satisfying the following three properties.

\begin{enumerate}
\item $V=\bigcup_{i=1}^nX_i.$ That is, each graph vertex is associated with
at least one tree vertex.

\item For every edge $(v, w)$ in the graph, there is a subset $X_i$ that
contains both $v$ and $w$. That is, vertices are adjacent in $G$ only when
the corresponding subtrees have a vertex in common.

\item If $X_i$, $X_j$ and $X_k$ are nodes, and $X_k$  is on the path from $%
X_i$ to $X_j,$ then $X_i\cap X_j \subseteq X_k$.
\end{enumerate}

Item (3) is equivalent to the fact that for each vertex $x$, the bags containing $x$
form a subtree of $T.$ The \textit{width} of a tree decomposition is the size of its largest set $%
X_i$ minus one. The \textit{treewidth} of a graph $G,$ written $tw(G),$ is
the minimum width among all possible tree decompositions of $G$. For more on treewidth,
see \cite{bod,diestel}.

Given an induced subgraph $H$ of $G,$ a \emph{cover} of $H$ in $G$, written
$\mathcal{C}_G(H)$ is a set of induced subgraphs $\{H_{i}:i\in I\}$ of $G$ whose
union contains $H$ (note that the subgraphs $H_i$ need not be disjoint). A
\emph{retract cover} of $H$ in $G$ is a cover where each $H_i$ is a retract; we write $\mathcal{C}_{R,G}(H)$ to denote a
retract cover.
See Figure~\ref{zero} for an example.
\begin{figure} [h!]
\begin{center}
\epsfig{figure=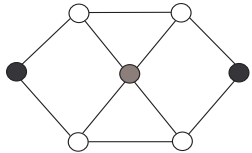,width=1.5in,height=1in}
\caption{The subgraph induced by the white and gray vertices forms a retract cover of the subgraph $H$ induced by the white vertices.}\label{zero}
\end{center}
\end{figure}

Define the \emph{retract cover cop number} of $H$ by
\begin{equation*}
\mathrm{rcc}_G(H)=\min_{\mathcal{C}_{R,G}(H)}\sum\limits_{H_{i}\in \mathcal{C}_{R,G}(H)}\mathrm{guard}%
(H_{i}),
\end{equation*}
where the minimum ranges over all retract covers $\mathcal{C}_{R,G}(H)$ of $H$ in $G.$ For example, $\mathrm{rcc}_G(H)=1$ in the
graph in Figure~\ref{zero}, while $\mathrm{rcc}_H(H)=2$. A retract cover which achieves this minimum is called a  \emph{minimal retract cover} of $H$. Note that if $\mathrm{rcc}_G(H)$-many cops are available, then after the finitely many rounds,
the cops can be positioned so that if the robber entered $H,$ he would be
immediately caught: for each retract in a retract cover of $H,$ the appropriate number of cops guard that retract.
This is an essential observation used to prove the following theorem. For a bag $B$, we use the
notation $\left\langle B\right\rangle$ for the subgraph induced by $B.$

\begin{theorem}
\label{main1}If $G$ is a graph, then%
\begin{equation*}
c(G)\leq 2\min_{T}\{\max_{B\in T}\mathrm{rcc}_G(\left\langle B\right\rangle )\},
\end{equation*}
where the minimum ranges over all tree decompositions $T$ of $G$.
\end{theorem}

For example, if $G$ is the $n\times n$ Cartesian grid $P_{n}\square P_{n}$,
then the following is a tree decomposition of $G$ into isometric paths.
Label the vertices as $(i,j)$, where $1\le i,j \le n.$ For $1\le i \le n-1$ and
$1\le j \le n,$ consider the path $$B_{i,j} = \{(i,k): j \le k \le n\} \cup \{(i + 1, k): 1 \le k \le j \}.$$
See Figure~\ref{one} for an illustration of one such path in the case $n=5$.
\begin{figure} [h!]
\begin{center}
\epsfig{figure=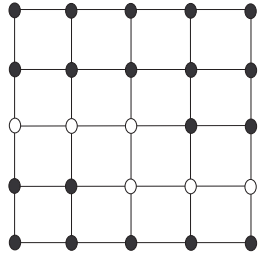,width=1.5in,height=1.5in}
\caption{The path with white vertices is the bag $B_{2,3}.$}\label{one}
\end{center}
\end{figure}
As the retract cover number of an isometric path is $1,$ this tree decomposition and Theorem~\ref{main1} gives an upper bound of $2$
for the cop number of grids (and of course, $2$ is the correct value).

\begin{proof}
Let $m=\min_{T}\{\max_{B\in T}rcc_G(\left\langle B\right\rangle )\},$ and let $%
T$ be a fixed tree decomposition of $G$ realizing the minimum. Place $2m$ arbitrarily cops in a fixed
bag $B$ of $T$ (any bag will do, or the cops can all move to a bag in the centre of the tree to
speed up capture). We call one team of $m$ cops $X$ and the other $X^{\prime }$. Match each cop $C$ from $X$ with a unique cop $C^{\prime }$ from $%
X^{\prime },$ so that the cops $C$ and $C^{\prime }$ share each other's
positions.  (In particular, at this phase of the cops' strategy, there are at least two cops at any position occupied by the cops).
After a finite number of rounds, the cops can position themselves on a minimal retract cover of $\left\langle B\right\rangle $ so
that $\left\langle B\right\rangle $ is guarded. Hence, $R$ cannot enter $B.$

Let $B^{\prime }$ be the unique bag adjacent to $B$ in $T$ which is on a shortest path connecting the
bag $B$ to a bag containing the robber. (Note that since the set
of bags containing the robber is a subtree of $T$, the robber is not necessarily in a unique bag. However, as $T$
is a tree, there is a unique bag neighbouring $B$ which has shortest distance to the subtree containing the robber.) The cops
would like to move to $B^{\prime }$ in such a way that $R$ cannot enter $B.$
The team $X$ of cops remains in the minimal retract cover of $\left\langle B\right\rangle $ and continue to guard it; the team $X^{\prime }$
moves to a  minimal retract cover of $\left\langle B^{\prime }\right\rangle$, and, after finitely many rounds guards $\left\langle B^{\prime }\right\rangle$.
Note that $B\cap B^{\prime }$ remains guarded throughout the transition. The cops in $X$ are now free to move without
concern that $R$ will enter $B.$

Now the cops in team $X$ can then move to a minimal retract cover $\mathcal{C}_{R,G}(H)$ of $\left\langle B''\right\rangle $ where $B''$ is the unique bag adjacent to $B'$ which is on a shortest path in $T$ to a bag containing the robber. Team $X$ then guards $\mathcal{C}_{R,G}(H)$. Note that we may now swap the roles of $X$ and $X'$, and we have moved the cops closer to the robber in the tree $T.$ We call the process of moving $X$ from a minimal retract cover of $B$ to a minimal retract cover of $B''$ a \emph{leap step}, as the cops in $X$ move through $B'$ and onwards toward $B''$. See Figure~\ref{leap}.
\begin{figure} [h!]
\begin{center}
\epsfig{figure=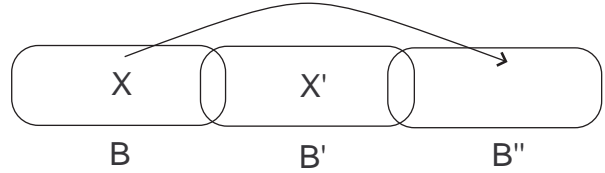}
\caption{A leap step}\label{leap}
\end{center}
\end{figure}

By the definition of tree decomposition, the bags containing $R$ form a
subtree $T^{\prime }$ of $T.$ By an iterated application of leap steps (after each such step, we swap the roles of teams $X$ and $X'$) for each bag on a shortest path connecting the cops' bag to the robbers, the cops move closer
to $T^{\prime }$, ensuring the robber will never enter bags previously guarded by the cops. By induction on the number of vertices of $T$, the cops may capture the robber.
\end{proof}

Observe that the proof of Theorem~\ref{main1} gives an algorithm for capturing the robber. Further, we can estimate
the length of the game using this algorithm. To be more precise, the
\emph{length} of a game is the number of rounds it takes (not including the
initial or $0$th round) to capture the robber. We say that a play of the game with $c(G)$
cops is \emph{optimal} if its length is the minimum over all possible strategies for the cops, assuming the robber is trying to evade capture for as
long as possible. There may be many optimal plays possible (for example, on $
P_{4},$ the cop may start on either vertex of the centre), but the length of
an optimal game is an invariant of $G.$ We denote this invariant $\mathrm{capt}(G),$
which we call the \emph{capture time} of $G.$ For a bound on the capture time in terms of the strategy in the proof of Theorem~\ref{main1}, note that the cops move to a bag in the centre of $T$, guard that bag, then move
towards the robber's bag.  Given a tree decomposition $T$, let the number of rounds it takes to guard a minimal retract cover of any bag in $T$ be at most $g_T$. Let $tr_T$ be the number of rounds it takes to move from a minimal retract cover of a bag $B$ to a minimal retract cover of a bag which is at most distance two from $B$ in $T$ (for instance, as in a leap step). Since the capture time of $T$ is $\left\lceil\mathrm{diam}(T)/2\right\rceil$ and the worst case is that the cops will need
to guard each bag and transition along each edge of a path with that length, a bound on the capture time of $G$ is then
\begin{equation}\label{two}
\mathrm{capt}(G) \le \min_T\{g_T(\left\lceil\mathrm{diam}(T)/2\right\rceil +1)+tr_T(\left\lceil\mathrm{diam}(T)/2\right\rceil)\},
\end{equation}
where the minimum ranges over all tree decompositions $T$ of $G.$

The bound (\ref{two}) may be far from the tight, as it depends on the values of the functions $g_T$ and $tr_T$. We can make a minor improvement on (\ref{two}) in the case $\mathrm{diam} (T)$ is odd (which implies that the centre of $T$ consists of two vertices). Start each of the two teams of cops on a minimal retract cover of different bags associated with the two bags of the centre of $T$. After each of these bags is guarded, the guards may proceed with leap steps.  Using this algorithm, a bound on the number of steps to capture the robber is found replacing the ceiling functions in (\ref{two}) with floor functions. Nevertheless, (\ref{two}) represents the first estimate on the capture time we are aware of applicable to diverse families of graphs such as outerplanar graphs.

We note that if each bag is a clique, then the idea of the proof shows a
strengthened bound.

\begin{theorem}\label{i}
If $G$ has a tree decomposition with each bag a clique, then the graph $G$
is cop-win.
\end{theorem}

\begin{proof}
The proof is analogous as before, but one cop is needed to guard a given
bag. That cop can move to $B\cap B^{\prime }$ without concern that $R$ will
enter $B.$
\end{proof}

Theorem~\ref{i} gives an alternative proof that chordal graphs are cop-win, as
chordal graphs are precisely those graphs with tree decompositions where
each bag induces a clique; see~\cite{bod}. Note also that $k$-trees are chordal, and have
treewidth $k;$ in particular, the bound (\ref{first}) is linear in $k,$
while Theorem~\ref{i} requires only one cop.

We finish with a bound in the case where there are conditions on the intersection of bags.

\begin{theorem}
\label{main2}If $G$ is a graph with a tree decomposition $T$ with the property that any two bags intersect in a clique, then%
\begin{equation*}
c(G)\leq \max_{B\in T}\mathrm{rcc}_G(\left\langle B\right\rangle )+1.
\end{equation*}
\end{theorem}

\begin{proof}
The proof is analogous to the proof of Theorem~\ref{main1}, except the team $%
X^{\prime }$ consists of one additional cop $C^{\prime }$. Using the
notation of the proof of Theorem~\ref{main1},  $m$-many cops guard $\left\langle
B\right\rangle ,$ while the cop $C^{\prime }$ moves to $B\cap B^{\prime }$.
Hence, $\left\langle B\right\rangle $ is guarded. The cops $X$ can move to a minimal retract cover of
$\left\langle B^{\prime }\right\rangle$ and $C^{\prime }$ ensures that $R$ never enters $
B $ (note that $B\cap B^{\prime }$ is a cut-set of $G).$ The cops may now iterate this procedure and, by induction
on the number of vertices of the tree decomposition, eventually capture the robber.
\end{proof}

\section{Acknowledgements}

The authors would like to thank BIRS where part of the research for this
paper was conducted.

\end{document}